\documentclass[paper=a4, fontsize=11pt]{scrartcl} 

\usepackage[T1]{fontenc} 
\usepackage{fourier} 
\usepackage[english]{babel} 
\usepackage{amsmath,amsfonts,amsthm} 
\usepackage{amsmath, calligra, mathrsfs}
\usepackage{amssymb}
\usepackage[mathscr]{euscript}
\usepackage{verbatim}
\usepackage{xcolor}
\usepackage{mdframed} 
\usepackage{soulutf8}
\usepackage{stmaryrd}
\usepackage{tikz-cd}
\usepackage{hyperref}
\usepackage{lipsum} 

\usepackage{sectsty} 
\allsectionsfont{\centering \normalfont\scshape} 

\usepackage{fancyhdr} 
\usepackage{xurl}
\newcommand*{\sheafhom}{\mathcal{H}\kern -.5pt om}
\numberwithin{equation}{section} 
\numberwithin{figure}{section} 
\numberwithin{table}{section} 

\newtheorem{thm}{Theorem}[section]
\newtheorem{cor}[thm]{Corollary}
\newtheorem{prop}[thm]{Proposition}

\theoremstyle{definition}
\newtheorem{defn}[thm]{Definition}

\newtheorem{exmp}[thm]{Example}

\theoremstyle{remark}
\newtheorem{rem}[thm]{Remark}

\DeclareMathOperator{\Var}{Var}

\DeclareMathOperator{\lk}{lk}

\DeclareMathOperator{\supp}{supp}

\DeclareMathOperator{\lcm}{lcm}

\DeclareMathOperator{\Conf}{Conf}

\DeclareMathOperator{\rk}{rank}

\newcommand{\horrule}[1]{\rule{\linewidth}{#1}} 

\title{	
	\normalfont \normalsize 
	\textsc{} \\ [25pt] 
	\horrule{0.5pt} \\[0.4cm] 
	\huge Simplicial chromatic polynomials as Hilbert series of Stanley--Reisner rings 

	\horrule{2pt} \\[0.5cm] 
}

\author{Soohyun Park} 

\date{\normalsize September 16, 2022} 

\begin{document}
	
	\maketitle 
	
	\begin{abstract}
		\noindent We find families of simplicial complexes where the simplicial chromatic polynomials defined by Cooper--de Silva--Sazdanovic \cite{CdSS} are Hilbert series of Stanley--Reisner rings of auxiliary simplicial complexes. As a result, such generalized chromatic polynomials are determined by $h$-vectors of auxiliary simplicial complexes. In addition to generalizing related results on graphs and matroids, the simplicial complexes used allow us to consider problems that are not necessarily analogues of those considered for graphs. Some examples include supports of cyclotomic polynomials and unimodality properties of polynomials involving $h$-vectorsl. \\
		
		\noindent If the $h$-vectors involed have sufficiently large entries, the Hilbert series are Hilbert polynomials of some $k$-algebra. As a consequence of connections between $h$-vectors and simplicial chromatic polynomials, we also find simplicial complexes whose $h$-vectors are determined by addition-contraction relations of simplicial complexes analogous to deletion-contraction relations of graphs. The constructions used involve generalizations of relations in Euler characteristics of configuration spaces and chromatic polynomials of graphs.
	\end{abstract}

	\section{Introduction}
	The main objective of this paper is to show that Euler characteristics of certain generalized configuration spaces are Hilbert series of Stanley--Reisner rings of associated simplicial complexes. The latter interpretation implies that these Euler characteristics are determined by $h$-vectors of some auxiliary simplicial complexes. In the course of doing this, we find that these Euler characteristics specialize to invariants satisfying the following properties:
	
	\begin{itemize}
		\item Polynomials with constant terms determining whether a cyclotomic polynomial of degree $n = p_1 \cdots p_d$ ($p_i$ distinct primes) has a nonzero term of degree $j$ for $0 \le j \le \varphi(n)$ (Corollary \ref{cyclcheck})

		\item Unimodality of polynomials involving $h$-vectors of simplicial complexes (Corollary \ref{unimodsimpchrom})

	\end{itemize}
	
	Each problem listed above is associated with an appropriate choice of simplicial complexes. \\
	
	Let $X$ be a manifold and $\Conf^n X = \{ (x_1, \ldots, x_n) \in X^n : x_i \ne x_j \text{ for all } i \ne j \}$. The configuration spaces we study generalize a relation between Euler characteristics of compactications of $\Conf^n X$ and chromatic polynomials of graphs. More specifically, our starting point is a combinatorial interpretation of the (compactly supported) Euler characteristic of the ordered configuration space $\Conf^n X$ and proper colorings of the complete graph with $n$ vertices $K_n$. 
	
	\begin{equation} \label{ordconfchi}
		\chi_c(\Conf^n X) = \chi_c(X)(\chi_c(X) - 1) \cdots (\chi_c(X) - (n - 1)).
	\end{equation}
	
	 Comparing this expression to the chromatic polynomial 
	 
	 \begin{equation} \label{cpletechrom}
	 	p_{K_n}(\lambda) = \lambda(\lambda - 1) \cdots (\lambda - (n - 1))
	 \end{equation}
	 
	 of the complete graph $K_n$ for $\lambda$ available colors, we find that $\chi_c(\Conf^n X) = p_{K_n}(\chi_c(X))$. \\
	
	We can consider how to generalize the relation $\chi_c(\Conf^n X) = p_{K_n}(\chi_c(X))$ to chromatic polynomials of arbitrary graphs $G$. Since adding edges to a graph $G$ introduces additional restrictions to proper colorings of $G$, it is natural to expect that an associated configuration space would be a partial compactification of $\Conf^n X$. In this context, we would allow the same points of $X$ to occupy slots corresponding to non-adjacent vertices. Indeed, such a generalization was found by Eastwood--Huggett (Theorem 2 on p. 155 of \cite{EaH}) for a modified configuration space parametrizing such configurations. \\
	
	Given a graph $G$ with vertices $V = \{ v_1, \ldots, v_n \}$, they consider Euler characteristics of the configuration space \[ M_G = M^n \setminus \bigcup_{e \in E} \Delta_e, \]
	
	for $M = \mathbb{CP}^{\lambda - 1}$, where \[ \Delta_e = \{ (x_1, \ldots, x_n) \in M^n : x_i = x_j \} \] for each edge $e = v_i v_j$ of $G$. The authors then show that 
	
	\begin{equation} \label{graphgen}
		p_G(\lambda) = \chi_c((\mathbb{CP}^{\lambda - 1})_G) \text{ (Theorem 2 on p. 1155 of \cite{EaH}),}
	\end{equation}
	
	which generalizes the correspondence between \ref{ordconfchi} and \ref{cpletechrom}. \\
	
	The construction of Eastwood--Huggett \cite{EaH} was further generalized with simplicial complexes replacing graphs in recent work of Cooper--de Silva--Sazdanovic \cite{CdSS}. They define the simplicial chromatic polynomial (Definition 6.1 on p. 738 of \cite{CdSS}, Definition \ref{simpchromdef}), which is the compactly supported Euler characteristic of a certain configuration space \ref{simpchromdef} which is a higher dimensional version of $M_G$ for simplicial complexes. 
	
	\begin{defn}  (Definition 2.1 on p. 725 and p. 738 of \cite{CdSS}) \\
		Let $S$ be a simplicial complex whose $0$-skeleton is given by the vertex set $V = V(S) = \{ v_1, \ldots, v_n \}$. Let $M$ be a topological space. For each simplex $\sigma = [ v_{i_1} \cdots v_{i_k} ]$, define the diagonal corresponding to $\sigma$ to be \[ D_\sigma = \{ (x_1, \ldots, x_n) \in M^n : x_{i_1} = \cdots  = x_{i_k} \}. \]
		
		We define the \textbf{simplicial configuration space} as 
		
		\begin{equation}  \label{simpconfdef}
			M_S = M^n \setminus \bigcup_{\sigma \in \Delta^V \setminus S} D_\sigma
		\end{equation}
		
		where $\Delta^V$ is the simplicial complex containing all subsets of the vertices $v_i$ (analogous to a simplex generated by independent vectors corresponding to the $v_i$) and $\Delta^V \setminus S$ denotes tuples of vertices in $V$ which do \emph{not} occur as simplices in $S$. 
	\end{defn}
	
	The simplicial chromatic polynomial $\chi_c(S)(t)$ associated to a simplicial complex $S$ is the compactly supported Euler characteristic of $M_S$ with $M = \mathbb{CP}^{t - 1}$ (Definition \ref{simpchromdef}). This polynomial is characterized (up to normalization) by an analogue of the deletion-contraction property for chromatic polynomials (Proposition \ref{adddelsimpchrom}, Corollary 6.1 and Proposition 6.4 on p. 738 of \cite{CdSS}). We would like to explore the combinatorial side of the simplicial chromatic polynomial. \\
	
	In Section 7 on p. 740 -- 742 of \cite{CdSS}, the authors show that the simplicial chromatic polynomial differs from a number of known polynomial invariants of graphs. They state that it is not a specialization of known polynomial invariants of graphs or simplicial complexes. However, we will show that $\chi_c(S)$ can often be built out of invariants of an auxiliary simplicial complex $T(S)$. The main tool we use to do this is the Stanley--Reisner ring (Definition \ref{srringdef}). 
	
	\begin{defn} (Stanley--Reisner ring, p. 53 -- 54 of \cite{Sta3}) \\
		Let $k$ be a field and $S$ be a simplicial complex with vertex set $V = \{ 1, \ldots, n \}$. We will call the subsets of $V$ belonging to $S$ \textbf{faces} and those that do \emph{not} belong to $S$ \textbf{nonfaces}. \\
		
		For a subset $A$ of the vertex set $V$, write $x_A = \prod_{i \in A} x_i \in k[x_1, \ldots, x_n]$. Let $I(A)$ denote the ideal in $k[x_1, \ldots, x_n]$ generated by all monomials $x_\sigma$ such that $\sigma \notin S$ (i.e. $\sigma \in \Delta^V \setminus S$ in the notation above). Note that $I(S)$ is generated by the \emph{minimal} nonfaces of $S$ since $\sigma$ being a nonface and $\alpha \supset \sigma$ implies that $\alpha$ is also a nonface of $S$. The quotient ring $k[S] := k[V]/I(S)$ is called the \textbf{Stanley--Reisner ring (or face ring)} of $S$. 
	\end{defn}
	
	The Stanley--Reisner ring $k[S]$ is the natural setting for many combinatorial problems (e.g. Stanley's proof of the upper bound theorem for simplicial spheres in Section 5.4 on p. 237 -- 240 of \cite{BH}). An overview of algebraic properties of this ring and their applications is given in a survey of Franscisco--Mermin--Schweig \cite{FMS}. Some examples are also given on p. 7 -- 8 of \cite{MS}. \\

	Our main result constructs families of simplicial complexes $S$ such that $\chi_c(S)$ can be expressed as a specialization of the Hilbert series of the Stanley--Reisner ring associated an auxiliary simplicial complex $T(S)$ (Theorem \ref{coverchromhilb}). \\
	
	The families of simplicial complexes that will consider satisfy certain intersection properties. We list the properties below and give examples of simplicial complexes satisfying them.
	
	\begin{defn} \label{propidef}
		Let $S$ be a simplicial complex with minimal nonfaces $\sigma_1, \ldots, \sigma_r$. A simplicial complex $S$ satisfies \textbf{property $\mathbf{I}$} if there is a collection of finite sets $\alpha_i$ such that $|\alpha_i| = |\sigma_i| - 1$ for each $1 \le i \le r$ and $\alpha_I \cap \alpha_p = \sigma_I \cap \sigma_p = \emptyset$ if $\sigma_I \cap \sigma_p = \emptyset$ and $|\alpha_I \cap \alpha_p| = |\sigma_I \cap \sigma_p| - 1$ if $|I| \ge 2$ and $\sigma_I \cap \sigma_p \ne \emptyset$ for each subset $I \subset [r]$ and $p \notin I$.
	\end{defn}
	
	Strictly speaking, the conditions required for the statement of Theorem \ref{coverchromhilb} to be satisfied come from intersection graphs of minimal nonfaces of simplicial complexes, which are defined in Definition \ref{conncpt} and used in the proof of Theorem \ref{coverchromhilb}. However, we phrase things in terms of the property $I$ condition since it makes it easier to understand what kinds of simplicial complexes satsify the conditions of Theorem \ref{coverchromhilb}.
	
	\begin{exmp}
		Here are two examples where property $I$ is satisfied. 
		
		\begin{itemize}
			\item The minimal nonfaces of $S$ are disjoint from each other by a simplicial complex $S$.
			
			\item There is a point $a \in V$ such that $\sigma_i \cap \sigma_j = \{ a \}$ for each $i, j$.
		\end{itemize}
		The general idea is that the minimal nonfaces either intersect at a small number of points or each $\sigma_i$ contains many points which are not contained in any $\sigma_j$ for $j \ne i$. Some specific simplicial complexes satisfying property $I$ are given in Example \ref{extexmp}.
	
	\end{exmp}

	\begin{thm} \label{coverchromhilb}
		Suppose that $S$ satisfies property $I$. Then, the simplicial chromatic polynomial $\chi_c(S, t)$ of $S$ is a normalization (in the sense of Proposition \ref{simpchromsrintcase}) of the Hilbert--Poincar\'e series of an auxiliary simplicial complex $T = T(S)$ which is determined by its $h$. Let $h_{T(S)}$ be the generating function of the $h$-vector of $T(S)$. This is the polynomial where the coefficient of $t^i$ is $h_i$. \\ 
		
		Tracing through the definitions, we have that \[ \chi_c(S)(t) = t^d(t - 1)^{n - d} h_{T(S)}(t^{-1}). \] Equivalently, we have that \[ \frac{\chi_c(S)(t)}{t^{d - m} (t - 1)^{n - d} } = t^m h_{T(S)}(t^{-1}), \] where $m = \dim T(S)$. This connects the simplicial chromatic polynomial to the reciprocal polynomial of $h_{T(S)}(t)$. Note that \emph{any} simplicial complex can have the same minimal nonfaces as the auxiliary simplicial complex $T(S)$ for some smiplicial complex $S$ satisfying property $I$.
	\end{thm}

	\begin{rem}
		
		Proposition \ref{simpchromsrintcase} gives an alternate set of conditions on the minimal nonfaces involving the number of connected components of a graph determined by the minimal nonfaces. For example, these conditions are satisfied by simplicial complexes where any pair of minimal nonfaces has a nonempty intersection. \\
	\end{rem}
	
 	In this setting of Theorem \ref{coverchromhilb}, the simplicial chromatic polynomial is entirely determined by the $h$-vector of $S$ (Corollary \ref{simpchromdet}). We also note that ``normalized'' version (without taking the reciprocal polynomial) of the simplicial chromatic polynomial is a specialization of Hilbert series of the canonical module of the Stanley--Reisner ring. \\

	Under suitable assumptions, there is a natural relation between simplicial chromatic polynomials and Hilbert \emph{polynomials} of other rings (Corollary \ref{simpchromhilbpol}) when the $h_i$ are sufficiently large.
	
	\begin{defn} \label{compdef}  
		Given a $k$-element subset $I = \{ \sigma_1, \ldots, \sigma_k \} \subset \Delta^V \setminus S$, let $G_I$ be the graph whose vertices are the $\sigma_i$ with two vertices corresponding to $\sigma_i, \sigma_j$ being connected by an edge if and only if $\sigma_i \cap \sigma_j \ne \emptyset$. Let $c(I)$ be the number of connected components of $G_I$. \\
	\end{defn}
	
	\begin{cor} \label{simpchromhilbpol}
	Let $S$ be a simplicial complex and $\sigma_1, \ldots, \sigma_r$ be the minimal nonfaces of $S$. Suppose that $c(I) = a$ for all subsets $I \subset [r]$ with $|I| \ge 2$ as in Proposition \ref{simpchromsrintcase}. If $h_{a + r} \ge 1$, $h_{a + 1}, h_{a + 2} \ge 3$, and $h_i \ge 1$ for all $i \ge a$, then \[ t^{-n} - \chi_c(S)(t^{-1}) + (t^{n + 1} - t^{n + a}) \sum_{\sigma_i} t^{-|\sigma_i|} = t^{-n} P(t^{-1}) \] for the Hilbert polynomial $P = P(x)$ of some $k$-algebra. \\
	
	If $a = 1$, this specializes to \[ t^{-n} - \chi_c(S)(t^{-1}) = t^{-n} P(t^{-1}). \]
	\end{cor} 
	
	These relations are analogous to some results on chromatic polynomials and Hilbert polynomials in the literature (e.g. Theorem 13 on p. 79 of \cite{Ste}, Proposition 3.3 on p. 9 of \cite{Bay}) although specializations of our results to these settings seem to yield different simplicial complexes. \\
	
	Finally, we show that there are some natural connections between simplicial chromatic polynomials and certain simplicial complexes whose structure depends on the coefficients of cyclotomic polynomials in Section \ref{applications}. The applications involved include supports of coefficients of cyclotomic polynomials (Corollary \ref{cyclcheck}) and unimodality of $h$-vectors (Corollary \ref{unimodsimpchrom}).
		
	\section*{Acknowledgements}
	I am very grateful to my advisor Benson Farb for his guidance and encouragement throughout the project. Also, I would like to thank him for extensive comments on preliminary versions of this paper. Finally, I am very thankful to Karim Adiprasito for helpful discussions on a project motivating the unimodality example studied here.

	\section{Simplicial chromatic polynomials and Hilbert series of Stanley--Reisner rings}
	
	In this section, we express simplicial chromatic polynomials as Hilbert series (Theorem \ref{coverchromhilb}, Proposition \ref{simpchromsrintcase}) and Hilbert polynomials (Corollary \ref{simpchromhilbpol}) of auxiliary simplicial complexes. \\
	
	\subsection{Transformation from Hilbert series of (auxiliary) simplicial complexes and $h$-vectors}
		
	Before showing the main reinterpretation of the simplicial chromatic polynomial, we first go over some basic definitions used on both sides of the correspondence. \\
	
	\begin{defn}  (Definition 2.1 on p. 725 and p. 738 of \cite{CdSS}) \\
		Let $S$ be a simplicial complex whose $0$-skeleton is given by the vertex set $V = V(S) = \{ v_1, \ldots, v_n \}$. Let $M$ be a topological space. For each simplex $\sigma = [ v_{i_1} \cdots v_{i_k} ]$, define the diagonal corresponding to $\sigma$ to be \[ D_\sigma = \{ (x_1, \ldots, x_n) \in M^n : x_{i_1} = \cdots  = x_{i_k} \}. \]
		
		We define the \textbf{simplicial configuration space} as 
		
		\begin{equation}
			 M_S = M^n \setminus \bigcup_{\sigma \in \Delta^V \setminus S} D_\sigma \label{simpconfdefuni}
		\end{equation}
		
		where $\Delta^V$ is the simplicial complex containing all subsets of the vertices $v_i$ (analogous to a simplex generated by independent vectors corresponding to the $v_i$) and $\Delta^V \setminus S$ denotes tuples of vertices in $V$ which do \emph{not} occur as simplices in $S$. 
	\end{defn}
	
	\begin{defn} \label{simpchromdef} (Definition 6.1 on p. 738 of \cite{CdSS}) \\
		Let $S$ be a simplicial complex and let $M$ be a manifold. Given $S$ and $M$, let \[ \chi_c(S, M) := \sum (-1)^k \rk H_c^k(M_S). \]
		
		The \textbf{simplicial chromatic polynomial} of a simplicial complex $S$ is the polynomial defined by the assignment $\chi_c(S) : t \mapsto \chi_c(S, \mathbb{CP}^{t - 1})$.
	\end{defn}
	
	Note that the union \ref{simpconfdefuni} is determined by the \emph{minimal} nonfaces $\sigma \in \Delta^V \setminus S$ since $\tau \subset \sigma \Rightarrow D_\tau \supset D_\sigma$. Under an appropriate normalization, the simplicial chromatic polynomial is uniquely defined by the following addition-contraction relation involving the minimal nonfaces: \\

	\begin{prop} \label{adddelsimpchrom} (Proposition 6.4 on p. 738, Definition 2.2 on p. 727 of \cite{CdSS}) \\
		Let $\sigma$ be a minimal nonface of a simplicial complex $S$, and 
		let $S_{/\sigma}$ be the tidied contraction, which removes every element sharing a vertex with $\sigma$ from $S$. \\
		
		The normalization $\chi_c(\Delta^{t - 1}, t) = t^n$ and the addition-contraction formula \[ \chi_c(S, t) - \chi_c(S \cup \{ \sigma \}, t) + \chi_c(S_{/\sigma}, t) = 0 \] determine a unique polynomial invariant of simplicial complexes.  \\     
	\end{prop}
	
	Given a manifold $M$, let $[M] = \chi_c(M)$. Using an inclusion-exclusion argument and the fact that $\chi_c(\mathbb{CP}^{t - 1}) = t$, we can obtain a more explicit expression for the simplicial chromatic polynomial with $\mathbb{CP}^{t - 1}$ substituted in for $M$ below. Note that $[M] = t$ if $M = \mathbb{CP}^{t - 1}$.

	\begin{align*}
		[M_S] &= \left[ M^n \setminus \bigcup_{\sigma \in \Delta^V \setminus S} D_\sigma \right] \\
		&= [M]^n - \left[ \bigcup_{\sigma \in \Delta^V \setminus S} D_\sigma  \right] \\
		&= [M]^n - \sum_{ \sigma_1, \ldots, \sigma_k \in I \subset [ \Delta^V \setminus S ] : |I| = k} (-1)^k [D_{\sigma_1} \cap \cdots \cap D_{\sigma_k}] 
	\end{align*}

	To simplify this final expression, we need to think about the number of ``independent values'' in an element of $D_{\sigma_1} \cap \cdots \cap D_{\sigma_k}$. This depends on the number of connected components of a certain graph (Definition \ref{compdef}).
	
	\begin{defn} \label{conncpt}
		Given a $k$-element subset $I = \{ \sigma_1, \ldots, \sigma_k \} \subset \Delta^V \setminus S$, let $G_I$ be the graph whose vertices are the $\sigma_i$ with two vertices corresponding to $\sigma_i, \sigma_j$ being connected by an edge if and only if $\sigma_i \cap \sigma_j \ne \emptyset$. Let $c(I)$ be the number of connected components of $G_I$.
	\end{defn}
	
	We can express the class $[ D_{\sigma_1} \cap \cdots \cap D_{\sigma_k} ]$ in $K_0(\Var_k)$ as a power of $[M]$ whose exponent depends on $c(I)$. More specifically, the class above simplifies to 
	
	\begin{equation} \label{inclexcleul}
		[M_S] = [M]^n - \sum_{ \sigma_1, \ldots, \sigma_k \in I \subset [ \Delta^V \setminus S ] : |I| = k} (-1)^k  [M]^{n - |\sigma_1 \cup \cdots \cup \sigma_k| + c(I)  }  
	\end{equation}
	
	for each simplicial complex $S$. Note that we can replace $\sigma \in \Delta^V \setminus S$ with \emph{minimal} nonfaces of $S$ in all the sums above. \\
	
	\begin{exmp}  \label{nonemptint}  (Nonfaces with pairwise nonempty intersections) \\ 
		If $\sigma_i \cap \sigma_j \ne \emptyset$ for all minimal nonfaces $\sigma_i, \sigma_j$ of $S$, we have that $c(I) = 1$ everywhere and
		
		\begin{equation} \label{nonemptinteqn}
			 [M_S] = [M]^n - \sum_{ \sigma_1, \ldots, \sigma_k \in I \subset [ \Delta^V \setminus S ] : |I| = k} (-1)^k  [M]^{n - |\sigma_1 \cup \cdots \cup \sigma_k| + 1.  }
		\end{equation}
		
		The individual terms of the sum \ref{nonemptinteqn} really only depend on $|\sigma_1 \cup \cdots \cup \sigma_k|$ and not the full information on the nonfaces $\sigma_1, \ldots, \sigma_k$ of $S$. An example where this occurs is Example 1.14 on p. 7 -- 8 of \cite{MS}. Note that this example applies since we can reduce to the case where the indices vary over \emph{minimal} nonfaces. It also gives a hint to a connection we will make with Stanley--Reisner rings which applies to arbitrary simplicial complexes.
	\end{exmp}

	 We would like to relate properties of this polynomial in $[M]$ with certain invariants of an algebraic structure associated to simplicial complexes which are parametrized by tuples of points which do \emph{not} belong to a specified simplicial complex as in the case of the simplicial chromatic polynomials above. 
	
	\begin{defn} \label{srringdef} (p. 201 -- 202 of \cite{Sta1}) \\
		Let $k$ be a field and $S$ be a simplicial complex. For a subset $A$ of the vertex set $V = \{ x_1, \ldots, x_n \}$, write $x_A = \prod_{x_i \in A} x_i \in k[x_1, \ldots, x_n]$. Let $I(A)$ denote the ideal in the polynomial ring $k[x_1, \ldots, x_n]$ generated by all monomials $x_\sigma$ such that $\sigma \notin S$ (i.e. $\sigma \in \Delta^V \setminus S$ in the notation above). \\
		
		Note that $I(S)$ is generated by the \emph{minimal} nonfaces of $S$ since $\sigma$ being a nonface and $\alpha \supset \sigma$ implies that $\alpha$ is also a nonface of $S$. The quotient ring $k[S] := k[V]/I(S)$ is called the \textbf{Stanley--Reisner ring (or face ring)} of $S$. 
	\end{defn}
	
	Throughout the proofs in this section, we implicitly use the following result to build simplicial complexes out of collections of finite sets.
	
	\begin{thm} \label{srcorr} (Stanley--Reisner correspondence, Definition 2.1, Definition 2.5, and Proposition 2.6 p. 211 -- 212  of \cite{FMS}) \\
		Let $X = \{ x_1, \ldots, x_n \}$. Each $x_i$ will be treated as the element $i$ of $\{ 1, \ldots, n \}$ when subsets are written below. Given a squarefree monomial ideal $I$, let \[ \Delta_I = \{ m \subset X : m \notin I \} \]
		be the simplicial complex consisting of squrefree monomials \emph{not} in $I$. \\
		
		Given a simplicial complex $\Delta$, let \[ I_\Delta = \langle m \subset X : m \notin \Delta \rangle. \]
		
		If $I$ is a squarefree monomial ideal, then $I_{\Delta_I} = I$. If $\Delta$ is a simplicial complex, then $\Delta_{I_\Delta} = \Delta$. \\ 
		
		The maps \[ \{ \text{squarefree monomial ideals} \}  \longrightarrow \{ \text{simplicial complexes} \}, I \mapsto \Delta_I \] and \[ \{ \text{simplicial complexes} \} \longrightarrow \{ \text{squarefree monomial ideals} \}, \Delta \mapsto I_\Delta \] induce a correspondence  \[ \{ \text{squarefree monomial ideals} \}  \longleftrightarrow \{ \text{simplicial complexes} \}.\]
	\end{thm}
	
	The connection of this ring with the simplicial chromatic polynomial comes from its Hilbert series. Let $S = k[x_1, \ldots, x_n]$ and let $J = \langle m_1, \ldots, m_r \rangle$ be a monomial ideal of $S$. The main idea is that the numerator of the Hilbert series $H(S/J, x)$ is the sum of the monomials which are \emph{not} in $J$ (top of p. 7 of \cite{MS}). One way to find these is to subtract the union of monomials which are in $J$ from the total set of \emph{all} monomials while counting the former using inclusion-exclusion. Normalizing (by division by $(1 - x_1) \cdots (1 - x_n)$), this gives the following polynomial in the numerator of the Hilbert series as shown on p. 1295 -- 1296 of \cite{GW}:
	
	\begin{equation} \label{hilbratl}
		H(S/J, x) = H(S, x) - H(J, x) = \sum_{\alpha \in \mathbb{N}^n} x^\alpha - \sum_{\substack{x^\alpha \in \langle m_I \rangle \\ I \subset [r] \\ |I| = 1}} x^\alpha + \sum_{\substack{x^\alpha \in \langle m_I \rangle \\ I \subset [r] \\ |I| = 2}} x^\alpha - \ldots + (-1)^r \sum_{\substack{x^\alpha \in \langle m_I \rangle \\ I \subset [r] \\ I = [r] }} x^\alpha, 
	\end{equation}
	
	where $[r] = \{ 1, \ldots, r \}$ and $m_I = \lcm(m_i : i \in I)$ for a subset $I \subset \{ 1, 2, \ldots, r \}$. In this case, the alternating sum of least common multiples of tuples of generators $m_i$ of $M$. Taking the $\lcm$ ensures that repeating indices aren't counted twice. Applying the normalization $(1 - x_1) \cdots (1 - x_n)$, \ref{hilbratl} can be rewritten as 
	
	\begin{equation} \label{hilbnorm}
		H(S/J, x) = \frac{1 - \sum_{\substack{I \subset [r] \\ |I| = 1}} m_I + \sum_{\substack{I \subset [r] \\ |I| = 2}} m_I - \ldots + (-1)^r \sum_{\substack{I \subset [r] \\ I = [r]}} m_I }{(1 - x_1) \cdots (1 - x_n)},
	\end{equation}
	
	 where $[r] = \{ 1, \ldots, r \}$ and $m_I = \lcm(m_i : i \in I)$ for a subset $I \subset \{ 1, 2, \ldots, r \}$ as above. Let $K(S/M, x)$ be the numerator of \ref{hilbnorm}. Since $m_I = \lcm(m_i : i \in I)$ and each of the $m_i$ are squarefree, we have that $\supp m_I = \bigcup_{i \in I} \supp m_i$. \\
	
	Putting everything together, we start finding the main relations between simplicial chromatic polynomials and Stanley--Reisner rings. We can start with the case where $c(I) = 1$ for all subsets $I \subset \Delta^V \setminus S$. 
	
	\begin{prop} \label{simpchromsrintcase}
		Given a manifold $M$, let $[M] = \chi_c(M)$. Let $S$ be a simplicial complex with vertex set $V = \{ 1, \ldots, n \}$ and $\sigma_1, \ldots, \sigma_r$ be the minimal nonfaces of $S$. If $c(I) = a$ for all $I \subset [f]$ with $|I| \ge 2$ (Definition \ref{compdef}), we have that 
		
		\begin{align*}
			\chi_c(S)(t) - t^n + (t^{n + 1} - t^{n + a}) \sum_{\sigma_i} t^{-|\sigma_i|} &= t^{n + a}( (1 - t^{-1})^{n - d} h_S(t^{-1}) - 1 ) \\
			&= t^{d + a} ( (t - 1)^{n - d} h_S(t^{-1}) - t^{n - d} ) \\
			\Longrightarrow \frac{\chi_c(S)(t) - t^n + t^{n + a} +  (t^{n + 1} - t^{n + a}) \sum_{\sigma_i} t^{-|\sigma_i|}   }{(t - 1)^{n - d}} &= t^{d + a} h_S(t^{-1}),
		\end{align*} 
		
		where the $\sigma_i$ are the minimal nonfaces of $S$. \\
		
		If $c(I) = 1$ for all $I \subset [n]$, then the formula above specializes to \[ \chi_c(S)(t) - t^n = t^{d + 1} ((t - 1)^{n - d} h_S(t^{-1} - t^{n - d}) ). \] Equivalently, we have that \[ \frac{\chi_c(S)(t) - t^n + t^{n + 1}}{t(t - 1)^{n - d}} = t^d h_S(t^{-1}). \] This connects the simplicial chromatic polynomial to the reciprocal polynomial of $h_S(t)$.
	\end{prop}

	\begin{rem}
		In Example \ref{nonemptint}, we have $c(I) = 1$ for all $I$. 
	\end{rem}
	
	\color{black} 
	\begin{proof}
		Since we take $x = (x_1, \ldots, x_n)$ in $K(S/J, x)$, the expression on the right hand side means substituting $x_i = [M]^{-1}$ for each $i \in [n]$. This is because the number of intersections taken in the $M_S$ setting and the number of simplices used in the inclusion-exclusion argument match up. Note that taking the $\lcm$ amounts to taking a union with the degree of each monomial in $H(S/M, x)$ giving the size of the set $|\sigma_1 \cup \cdots \cup \sigma_k|$.
	\end{proof}
	
	The following result implies that the ``interesting'' part of the simplicial chromatic polynomial (i.e. the term obtained from subtracting the leading term) only depends on the $h$-vector of the simplicial complex: \\
	
	\begin{prop} \label{hilbpoinhvec} (Corollary 1.15 on p. 8 of \cite{MS}) \\
		Let $f_i$ be the number of $i$-faces in the simplicial complex $S$. Then, \[ H(S/I(S); t, \ldots, t) = \frac{1}{(1 - t)^n} \sum_{n = 0}^d f_{i - 1} t^i (1 - t)^{n - i} =  \frac{h_0 + h_1 t + h_2 t^2 + \ldots + h_d t^d}{(1 - t)^d}, \]
		
		where $d = \dim S + 1$.
	\end{prop}

	There are analogues of the results above which we can give for general simplicial complexes. For a general simplicial complex, terms of the sum describing the class $[M_S]$ in $K_0(\Var_k)$ corresponding to a subset $I = \{ \sigma_1, \ldots, \sigma_k \} \subset \Delta^V \setminus S$ depends on more than $|\sigma_1 \cup \cdots \cup \sigma_k|$. 
	
	We can find a counterpart of Proposition \ref{simpchromsrintcase} for simplicial complexes satisfying a (mild) assumption on finite set covers. This enables us to express the nontrivial part of $[M_S]$ in terms of the Hilbert--Poincar\'e series of the Stanley--Reisner ring of an auxiliary simplicial complex depending on $S$ (Theorem \ref{coverchromhilb}).

	\begin{proof} (of Theorem \ref{coverchromhilb}) \\
		Compare the terms of the expansion \[ [M_S] = [M]^n - \sum_{ \sigma_1, \ldots, \sigma_k \in I \subset [ \Delta^V \setminus S ] : |I| = k} (-1)^k  [M]^{n - |\sigma_1 \cup \cdots \cup \sigma_k| + c(I)  }   \] with those of \[ H(S/J, x) = \frac{1 - \sum_{\substack{I \subset [r] \\ |I| = 1}} m_I + \sum_{\substack{I \subset [r] \\ |I| = 2}} m_I - \ldots + (-1)^r \sum_{\substack{I \subset [r] \\ I = [r]}} m_I }{(1 - x_1) \cdots (1 - x_n)}, \]
		
		where $[r] = \{ 1, \ldots, r \}$ and $m_I = \lcm(m_i : i \in I)$ for a subset $I \subset \{ 1, 2, \ldots, r \}$. Given a simplicial complex $T$, let $d = \dim T + 1$. Recall that
		
		\begin{equation} \label{hvecfvechilb}
			H(T/I(T); t, \ldots, t) = \frac{1}{(1 - t)^n} \sum_{n = 0}^d f_{i - 1} t^i (1 - t)^{n - i} =  \frac{h_0 + h_1 t + h_2 t^2 + \ldots + h_d t^d}{(1 - t)^d} 
		\end{equation}
		
      	by Proposition \ref{hilbpoinhvec}. \\
      	
      	In the sum \ref{hvecfvechilb} (which results from the specialization $x_1 = x_2 = \cdots  = x_n = t$), terms of degree $u$ correspond to $\beta_1, \ldots, \beta_k \in \Delta^V \setminus T$ such that $|\beta_1 \cup \cdots \cup \beta_k| = u$ since $\lcm(x^{\beta_1}, \ldots, x^{\beta_k}) = x^{\beta_1 \cup \cdots \cup \beta_k}$. Note that substituting $x_1 = x_2 = \cdots  = x_n = t$ into $x^{\beta_1 \cup \cdots \cup \beta_k}$ gives $t^{|\beta_1 \cup \cdots \cup \beta_k|}$. For $i \in [r]$, each $m_i$ in $J = \langle m_1, \ldots, m_r \rangle$ from \ref{hilbratl} corresponds to a subset of $[n]$ giving a minimal nonface of the simplicial complex $T$. The class $[M_S]$ is more closely related to the expression we obtain by replacing $t$ with $t^{-1}$ in $H(T/I(T); t, \ldots, t)$ in a way analogous to Proposition \ref{simpchromsrintcase}. \\
      	
      	The idea is to find subsets $\alpha_1, \ldots, \alpha_r \subset [n]$ such that $|\sigma_{i_1} \cup \cdots  \cup \sigma_{i_k}| - c(I) = |\alpha_{i_1} \cup \cdots \cup \alpha_{i_k}|$. Setting $\sigma_I = \sigma_{i_1} \cup \cdots \cup \sigma_{i_k}$ and $\alpha_I = \alpha_{i_1} \cup \cdots \cup \alpha_{i_k}$ for $I = \{ i_1, \ldots, i_k \}$, the claim can be rewritten as $|\sigma_I| - c(I) = |\alpha_I|$.  As a polynomial, this would let us use the sum from Proposition \ref{hilbpoinhvec} if the $\alpha_1, \ldots, \alpha_r$ give the Stanley--Reisner ideal of some simplicial complex. By the Stanley--Reisner correspondence (Theorem \ref{srcorr}), such a simplicial complex does exist. The auxiliary simplicial complex $T(S)$ in the statement of Theorem \ref{coverchromhilb} would be the simplicial complex whose minimal nonfaces correspond to the $\alpha_i$ via the Stanley--Reisner correspondence. Thus, the problem is reduced to showing that the assumptions in the statement allow the existence of such $\alpha_i$. \\
      	
      	We use induction on $|I|$ to show that the sets $\alpha_i$ in the definition of property $I$ (Definition \ref{propidef}) satisfy the relation $|\sigma_I| - c(I) = |\alpha_I|$. Since $c(I) = 1$ if $|I| = 1$, the case where $|I| = 1$ is really the statement that $|\alpha_i| = |\sigma_i| - 1$. Suppose that $|\alpha_I| = |\sigma_I| - c(I)$. Given $p \notin I$, let $J = I \cup \{ p \}$. We would like to show that $|\alpha_{J_p}| = |\sigma_{J_p}| - c(J_p)$. Since $\sigma_p$ adds a new component if and only if $\sigma_p \cap \sigma_i = \emptyset$ for all $i \in I$, we have that 
      	
      	\[ c(J_p) = \begin{cases} 
      		c(I) & \text{ if } \sigma_i \cap \sigma_p \ne \emptyset \text{ for some $i \in I$ } \\
      		c(I) + 1 & \text{ if } \sigma_i \cap \sigma_p = \emptyset \text{ for all $i \in I$ } 
      	\end{cases}
      	\]
      	
      	Using the notation $\sigma_I = \bigcup_{i \in I} \sigma_i$, this simplifies to 
      	
      	\[ c(J_p) = \begin{cases} 
      		c(I) & \text{ if } \sigma_I \cap \sigma_p \ne \emptyset \\
      		c(I) + 1 & \text{ if } \sigma_I \cap \sigma_p = \emptyset.  
      	\end{cases}
      	\]
      	
      	If $\sigma_I \cap \sigma_p = \emptyset$, then $|\sigma_{J_p}| = |\sigma_I| + |\sigma_p|$ and
      	
      	\begin{align*}
      		|\sigma_{J_p}| - c(J_p) &= |\sigma_I| + |\sigma_p|  - c(I) - 1 \\
      		&= (|\sigma_I| - c(I)) + (|\sigma_p| - 1) \\
      		&= |\alpha_I| + |\alpha_p| \\
      		&= |\alpha_{J_p}|,
      	\end{align*}
      	
      	where the last line follows from the assumption that $\alpha_I \cap \alpha_p = \emptyset$ if we have $I \subset [r]$ and $p \notin I$ such that $\sigma_I \cap \sigma_p = \emptyset$ and we are considering the case where $\sigma_I \cap \sigma_p = \emptyset$ (which means that $\alpha_I \cap \alpha_p = \emptyset$). \\
      	
      	Now consider the case where $\sigma_I \cap \sigma_p \ne \emptyset$. In this case, we have that $c(J_p) = c(I)$. Note that $|\sigma_{J_p}| = |\sigma_I| + |\sigma_p| - |\sigma_I \cap \sigma_p|$. Similarly, we have that 
      	
      	\begin{align*}
      		|\alpha_{J_p}| &= |\alpha_I| + |\alpha_p| - |\alpha_I \cap \alpha_p| \\ 
      		&= |\sigma_I| - c(I) + |\sigma_p| - 1 - |\alpha_I \cap \alpha_p| \\
      		&= |\sigma_I| - c(I) + |\sigma_p| - 1 - |\sigma_I \cap \sigma_p| + 1 \\
      		&= |\sigma_I| + |\sigma_p| - |\sigma_I \cap \sigma_p| - c(I) \\
      		&= |\sigma_{J_p}| - c(I),
      	\end{align*} 
      	
      	where we used the assumption that $|\alpha_I \cap \alpha_p| = |\sigma_I \cap \sigma_p| - 1$ if $\sigma_I \cap \sigma_p \ne \emptyset$ in the third line. Thus, the desired conclusion follows from induction on the size of $I \subset [r]$. 
      	
	\end{proof}

	\begin{exmp} \label{extexmp}
		We give some examples where the assumptions of Theorem \ref{coverchromhilb} hold. Note that it suffices to start with an initial collection of subsets $\sigma_i \subset [n]$ (taken to be generators of a squarefree monomial ideal) since they can always be taken to be the minimal nonfaces of some simplicial complex by the Stanley--Reisner correspondence (p. 212 of \cite{FMS}).  \\
		
		\begin{enumerate}
			\item Suppose that $r = 2$ (i.e. that there are two minimal nonfaces) and $\sigma_1 \cap \sigma_2 = \emptyset$. For example, this applies to the simplicial complex $\Delta$ corresponding to the square graph $abcd$ for the vertex set $a, b, c, d$, which has $I(\Delta) = \langle ac, bd \rangle$ (Example 1.15 on p. 2 of \cite{Stu}). Then, we can simply take $\alpha_i$ to be any subset of $\sigma_i$ with one element removed. Since $\alpha_1 \subset \sigma_1$ and $\alpha_2 \subset \sigma_2$, we have that $\alpha_1 \cap \alpha_2 = \emptyset$. If we take $I = \{ 1, 2 \}$, then $|\sigma_I \cap \sigma_i| = |\sigma_i|$ and $|\alpha_I \cap \alpha_i| = |\alpha_i| = |\sigma_i| - 1$ and the assumptions of Theorem \ref{coverchromhilb} are satisfied. These arguments apply in general when $\sigma_i \cap \sigma_j$ for $i \ne j$. 
			
			\item Let $r = 3$ and take subsets $\sigma_1, \sigma_2, \sigma_3 \subset [n]$ of size $\ge 2$ such that $\sigma_i \cap \sigma_j = \sigma_1 \cap \sigma_2 \cap \sigma_3 = \{ a \}$ for some point $a$. Let $\alpha_i = \sigma_i \setminus \{ a \}$. Since the $|I| = 3$ case is trivial, it suffices to consider the cases $|I| = 1$ and $|I| = 2$. The $|I| = 1$ case has to do with pairwise intersections. This works since $|\alpha_i \cap \alpha_j| = 0 = |\sigma_i \cap \sigma_j| - 1$ for $i \ne j$. As for the $|I| = 2$ case, the nontrivial instance is from noting that $|\alpha_1 \cap \alpha_2 \cap \alpha_3| = 0 = |\sigma_1 \cap \sigma_2 \cap \sigma_3| - 1 = 1$. This satisfies the assumptions of Theorem \ref{coverchromhilb}. Adding a subset $\sigma_4$ such that $\sigma_4 \cap \sigma_i = \emptyset$ for $i = 1, 2, 3$ still gives a collection of subsets satisfying the assumptions of Theorem \ref{coverchromhilb}. 
			
			\item Another example with $r = 3$ uses $\sigma_1, \sigma_2, \sigma_3 \subset [n]$ such that $|\sigma_1 \cap \sigma_2| = |\sigma_1 \cap \sigma_3| = |\sigma_2 \cap \sigma_3| = 1$ and $\sigma_1 \cap \sigma_2 \cap \sigma_3 = \emptyset$. Let $a = \sigma_1 \cap \sigma_2$, $b = \sigma_1 \cap \sigma_3$, and $c = \sigma_2 \cap \sigma_3$. Then, take $\alpha_1 = \sigma_1 \setminus \{ a \}$, $\alpha_2 = \sigma_2 \setminus \{ c \}$, and $\alpha_3 = \sigma_3 \setminus \{ b \}$. As mentioned in the previous example, the conditions with $|I| = 3$ are trivial. It remains to consider the cases with $|I| = 1$ and $|I| = 2$. We are done with the $|I| = 1$ case since the $\alpha_i$ are pairwise disjoint. The nontrivial part of the $|I| = 2$ case comes from noting that $\sigma_1 \cap \sigma_2 \cap \sigma_3 = \emptyset$ and $\alpha_1 \cap \alpha_2 \cap \alpha_3 = \emptyset$. 
		\end{enumerate}
	\end{exmp}
	
	\pagebreak 
	
	\color{black}
	\begin{rem} \label{transform} ~\\
		\vspace{-3mm}
		\begin{enumerate}
			\item 
			If we don't invert the variables in $\chi_c(S, t)$ and $S$ is a Cohen--Macaulay simplicial complex, the expression obtained for the simplicial chromatic polynomial $\chi_c(S, t)$ is also closely related to a specialization of the Hilbert--Poincar\'e series of the *canonical module of $k[S]$ (Exercise 5.6.6 on p. 246 of \cite{BH}): \[ H_{\omega_{k[S]}}(t) = \sum_{F \in S} \dim_k \widetilde{H}_{\dim \lk F} (\lk F; k) \prod_{X_i \in F} \frac{t_i}{1 - t_i} = (-1)^d H_{k[S]}(t_1^{-1}, \ldots, t_n^{-1}) \]
			
			The intermediate expression $\lk F = \{ G : F \cup G \in S, F \cap G = \emptyset \}$ (Definition 5.3.4 on p. 232 of \cite{BH}) is also involved in determining when a simplicial complex is Cohen--Macaulay (Theorem 12.27 on p. 211 of \cite{Sta1}). There are even simpler relations with $H_{k[S]}(t_1, \ldots, t_n) = (-1)^d H_{k[S]}(t_1^{-1}, \ldots, t_n^{-1})$ if the simplicial complexes involved are Euler complexes (Exercise 5.6.7 on p. 246 of \cite{BH}) and $t^a H_{k[S]}(t_1, \ldots, t_n) = (-1)^d H_{k[S]}(t_1^{-1}, \ldots, t_n^{-1})$ if they are Gorenstein complexes (Exercise 5.6.9 on p. 246 of \cite{BH}).

			\item The Hilbert series of auxiliary simplicial complexes determined by the simplicial chromatic polynomials in the setting of Proposition \ref{simpchromsrintcase} and Theorem \ref*{coverchromhilb} can be used to generate Hilbert series of other rings using operations such as addition, partial sums, and multiplication related to joins of simplicial complexes as indicated in Corollary 6.6 on p. 738 of \cite{CdSS} (Proposition 2.4 on p. 131 of \cite{Bre}).
		\end{enumerate}

	\end{rem}
	
	We can combine Proposition \ref{hilbpoinhvec} with the main results (Proposition \ref{simpchromsrintcase} or Theorem \ref{coverchromhilb}) and the defining property of the simplicial chromatic polynomial (Proposition \ref{adddelsimpchrom}) to make an observation on simplicial chromatic polynomials of simplicial complexes $S$ in the setting of the results above.
	
	\begin{cor} \label{simpchromdet}
		Suppose that $S$ is a simplicial complex satisfying the assumptions of Proposition \ref{simpchromsrintcase} or Theorem \ref{coverchromhilb}.
		
		\begin{enumerate}
			
			\item The simplicial chromatic polynomial $\chi_c(S, t)$ is entirely determined by the $h$-vector of $S$. 
			
			\item The $h$-vector of the associated simplicial complex ($S$ in Proposition \ref{simpchromsrintcase} and $T = T(S)$ in Theorem \ref*{coverchromhilb}) is entirely determined by the defining addition-contraction relation and normalization for simplicial chromatic polynomials from Proposition 6.4 on p. 738 of \cite{CdSS}.
		\end{enumerate}

	\end{cor}

	\subsection{Simplicial chromatic polynomials as Hilbert polynomials of other rings}

	We can also connect the original simplicial chromatic polynomials $\chi_c(S, t)$ associated to a simplicial complex $S$ to Hilbert series of rings without making any modifications on the latter ring. More specifically, we will work with conditions under which a polynomial is a Hilbert polynomial (i.e. ``approximately a Hilbert series'' -- see p. 131 and Theorem 2.2 on p. 129 of \cite{Bre}) which are listed below.

	\begin{prop} \label{hilbpolycrit} (Corollary 3.10 on p. 138 of \cite{Bre}) \\
		Let $P(x) = \sum_{i = 0}^d a_i x^i$. If $a_0, \ldots, a_d \in \mathbb{N}$ and $a_1, a_2 \ge 3$, then $P(x)$ is a Hilbert polynomial associated to some  $k$-algebra.
	\end{prop}
	
	Combining Proposition \ref{hilbpolycrit} with the inclusion-exclusion arguments used earlier, we find some conditions under which $[M_S] - [M]$ is a Hilbert polynomial with $[M]$ substituted in for the variable (Corollary \ref{simpchromhilbpol}).

	\begin{proof} (of Corollary \ref{simpchromhilbpol}) \\
		The expansion \[ [M_S] = [M]^n - \sum_{\substack{\sigma_1, \ldots, \sigma_k \in I \subset [ \Delta^V \setminus S ] \\ |I| = k}} (-1)^k  [M]^{n - |\sigma_1 \cup \cdots \cup \sigma_k| + c(I)  }  \] 
		
		implies that \[ [M]^n - [M_S] = [M]^n \sum_{\substack{\sigma_1, \ldots, \sigma_k \in I \subset [ \Delta^V \setminus S ] \\ |I| = k}} (-1)^k  [M]^{- |\sigma_1 \cup \cdots \cup \sigma_k| + c(I)  } . \]
		
		Replacing $t$ with $t^{-1}$, we find that \[ t^{-n} - \chi_c(S, t^{-1}) + (t^{n + 1} - t^{n + a}) \sum_{\sigma_i} t^{-|\sigma_i|} = t^{-n} \sum_{\substack{\sigma_1, \ldots, \sigma_k \in I \subset [ \Delta^V \setminus S ] \\ |I| = k}} (-1)^k  t^{- (|\sigma_1 \cup \cdots \cup \sigma_k| - c(I))  }. \]
		
		The coefficient of $t^{-r}$ on the right hand side is \[ \sum_{\substack{ \sigma_1, \ldots, \sigma_k \in I \subset [ \Delta^V \setminus S ] \\ |I| = k \\ |\sigma_1 \cup \cdots \cup \sigma_k| - c(I) = r  }} (-1)^k.  \]
		
		Since we assumed that $c(I) = a$ for all $I$ with $|I| \ge 2$, the coefficient of $t^{-r}$ is 
		
		\begin{equation} \label{coeff1}
			\sum_{\substack{ \sigma_1, \ldots, \sigma_k \in I \subset [ \Delta^V \setminus S ] \\ |I| = k \\ |\sigma_1 \cup \cdots \cup \sigma_k| = a + r  }} (-1)^k.
		\end{equation}
		
		Recall that 
		
		\begin{equation} \label{coeff2}
			h_s = \sum_{\substack{ \sigma_1, \ldots, \sigma_k \in I \subset [ \Delta^V \setminus S ] \\ |I| = k \\ |\sigma_1 \cup \cdots \cup \sigma_k| = s }} (-1)^k
		\end{equation}
		
		by a correspondence between the numerators of \ref{hilbnorm} and \ref{hvecfvechilb}. \\
		
		Comparing \ref{coeff2} to the coefficient of $t^{-r}$ given by \ref{coeff1}, we have that the coefficient of $t^{-r}$ is equal to $h_{a + r}$. Combining the assumptions that $h_{a + r} \ge 1$, $h_{a + 1}, h_{a + 2} \ge 3$, and $h_i \ge 0$ for all $i \ge a$ with Proposition \ref{hilbpolycrit}, we find that $t^{-n} - \chi_c(S, t^{-1}) + (t^{n + 1} - t^{n + a}) \sum_{\sigma_i} t^{-|\sigma_i|} = t^{-n} P(t^{-1})$ for some Hilbert polynomial $P = P(x)$.

	\end{proof}

	\section{Applications} \label{applications}
	
	In this section, we outline applications of the main results to properties of simplicial chromatic polynomials in various contexts. This includes properties that generalize those of chromatic polynomials or matroids. The fact that we work with simplicial complexes are not necessarily independence complexes of graphs allows us to obtain symmetric properties in other contexts. Here are the main types of applications considered:
	
	\begin{itemize}
		\item The support of cyclotomic polynomials (Corollary \ref{cyclcheck})

		\item Modifications yielding unimodal coefficients (Corollary \ref{unimodsimpchrom})
		
	\end{itemize}
	
	\subsection{Connection to cyclotomic polynomials} \label{conncyclsec}
	
	We consider simplicial complexes whose topology encodes information on coefficients of cyclotomic polynomials.
	
	\begin{defn} (p. 114 of \cite{MR}) \\
		Let $K_p$ denote the $0$-dimensional abstract simplicial complex consisting of $p$ vertices. Given $d$ distinct primes $p_1, \ldots, p_d$, we set \[ K_{p_1, \ldots, p_d} = K_{p_1} * \cdots * K_{p_d}, \]
		
		where $*$ denotes join. Setting $n = p_1 \cdots p_d$, we note that each facet of $K_{p_1, \ldots, p_d}$ corresponds to a residue $\pmod{n}$ by the Chinese Remainder Theorem. \\
		
		Given a subset $A \subset \{ 0, 1, \ldots, \varphi(n) \}$, let $K_A$ be the subcomplex of $K_{p_1, \ldots, p_d}$ containing the entire $(d - 2)$-simplex whose facets correspond to elements of \[ A \cup \{ \varphi(n) + 1, \varphi(n) + 2, \ldots, n \}. \]
	\end{defn}
	
	Here is the result of Musiker--Reiner \cite{MR} connecting this simplicial complex to the coefficients of simplicial chromatic polynomials. \\
	
	\begin{thm} \label{cycltop} (Musiker--Reiner, Theorem 1.1 on p. 114 of \cite{MR}) \\
		For a squarefree postiive integer $n = p_1 \cdots p_d$, with cyclotomic polynomial \[ \Phi_n(x) = \sum_{j = 0}^{\varphi(n)} c_j x^j, \] one has \[ \widetilde{H}_i(K_{\{j\}} ; \mathbb{Z} ) =  \begin{cases} 
			\mathbb{Z}/c_j \mathbb{Z} & \text{ if $i = d - 2$} \\
			\mathbb{Z} & \text{ if both $i = d - 1$ and $c_j = 0$} \\
			0 & \text{ otherwise.} 
		\end{cases}
		\]
		
		for each $0 \le j \le \varphi(n)$.
	\end{thm}

	While we only account for non-torsion components in the Euler characteristic, the Euler characteristic associated with this simplicial complex still records information on the \emph{parity} of nonzero coefficients. \\

	Before outlining the connection of Theorem \ref{coverchromhilb} with cyclotomic polynomials via Theorem \ref{cycltop}, we define some notation.

	\begin{defn}
		Given a simplicial complex $S$, let $h_S(t) = h_0 + h_1 t^2 + \ldots + h_d t^d$, where $h = (h_0, \ldots, h_d)$ is the $h$-vector associated to $S$.
	\end{defn}
	
	There are simplicial complexes satisfying the intersection conditions in the main result.
	
	\begin{exmp}
		There are simplicial complexes $K_{p_1, \ldots, p_d}$ such that $K_{ \{ j \} }$ satisfies the conditions of Theorem \ref{coverchromhilb}. This means that their simplicial complexes are determined by $h$-vectors of auxiliary simplicial complexes. For example, consider the subcomplex $K_{ \{ 1 \} } \subset K_{2, 3}$. Writing $K_{2, 3} = K_2 * K_3$, let $K_2 = \{ a, b \}$ and $K_3 = \{ c, d, e \}$. The minimal nonfaces of $K_{ \{ 1 \} }$ are $\beta_1 = \{ a, b \}, \beta_2 = \{ c, d \}, \beta_3 = \{ c, e \}, \beta_4 = \{ d, e \},$ and $\beta_5 = \{ a, e \}$. The nonfaces $\beta_1, \beta_2, \beta_3, \beta_4$ correspond to pairs of points from the same vertex set $K_i$ and the nonface $\beta_5$ is the facet that was removed in the construction of $K_{ \{ 1 \} }$. Setting $\sigma_i = \beta_i$ for each $1 \le i \le 5$ in Theorem \ref{coverchromhilb}, we can set $\alpha_1 = \{ b \}, \alpha_2 = \{ c \}, \alpha_3 = \{ d \}, \alpha_4 = \{ e \},$ and $\alpha_5 = \{ a \}$. 
	\end{exmp}
	
	The simplicial complexes $K_{ \{ j \} }$ are also connected to simplicial chromatic polynomials whose constant terms detect whether a cyclotomic polynomial contains a term of a given degree.
	
	\begin{cor} \label{cyclcheck}
		Given a cyclotomic polynomial $\Phi_n(x) = \sum_{j = 0}^{\varphi(n)} c_j x^j$, there are simplicial complexes $S_j$ such that the constant term of $\frac{\chi_c(S)(t)}{t^{d - m} (t - 1)^{n - d} }$ is equal to $1 + (-1)^d$ or $(-1)^d$ depending on whether $c_j = 0$ or $c_j \ne 0$ respectively. This means that $\chi_c(S_j)$ determines whether $c_j = 0$ or not.
	\end{cor}
	
	\begin{proof}
		Recall that our Theorem \ref{coverchromhilb} states that \[ \chi_c(S)(t) = t^d(t - 1)^{n - d} h_{T(S)}(t^{-1}) \Longrightarrow  \frac{\chi_c(S)(t)}{t^{d - m} (t - 1)^{n - d} } = t^m h_{T(S)}(t^{-1}), \] where $m = \dim T(S)$ for simplicial complexes $S$ satisfying appropriate intersection properties (Definition \ref{propidef}). \\
		
		If we add an extraneous element (vertex) not in the ground set of $K_{ \{ j \} }$ to every minimal nonface of $K_{\{ j \}}$, then we obtain the minimal nonfaces of a simplicial complex $S$ such that the auxiliary simplicial complex in Theorem \ref{coverchromhilb} is $T(S) = K_{ \{ j \} }$.  In other words, we consider the simplicial complex whose minimal nonfaces are $\sigma_i = \alpha_i \cup \{ q \}$ for some minimal nonface $\alpha_i$ of $K_{ \{ j \} }$ and extraneous vertex $q$ \emph{not} in the ground set/vertex set of $K_{ \{ j \} }$. Then, $\sigma_I \cap \sigma_p \ne \emptyset$ for any choice of $I$ and $p$ in Definition \ref{propidef} and removing the extra vertex $q$ would decrease the size of the intersection $\sigma_I \cap \sigma_p$ by 1 since it is contained in every minimal nonface of the simplicial complex $S$. The following result implies that the constant term of $\frac{\chi_c(S)(t)}{t^{d - m} (t - 1)^{n - d} }$ is equal to $1 + (-1)^d$ or $(-1)^d$ depending on whether $c_j = 0$ or $c_j \ne 0$ respectively. \\
		
		\begin{prop} (Corollary 5.1.9 on p. 213 of \cite{BH}) \\
			Given a simplicial complex $S$, we have that $h_d = (-1)^{d - 1}(\chi(S) - 1)$.
		\end{prop}
		
	\end{proof}
	
	\begin{rem} \label{cyclcheckext}
		The construction used in the proof of Corollary \ref{cyclcheck} can be repeated to construct $\sigma_i$ satisfying the intersection property in Definition \ref{propidef} given any collection of sets $\alpha_i$.
	\end{rem}

	\subsection{Unimodality and modified simplicial complexes} \label{logconsec}

	Another way to interpret the simplicial chromatic polynomial is to interpret it as the characteristic polynomial of a diagonal/hypergraph linear subspace arrangement. Note that recent results of Pagaria--Pezzoli \cite{PP} and Crowley--Huh--Larson--Simpson--Wang \cite{CHLSW} giving a Chow ring for polymatroids imply that $\chi_c(S)$ (and therefore $h_{T(S)}$) is determined by a Chow ring structure. This does \emph{not} imply that this polynomial has log concave coefficients (Remark 6.6 on p. 35 of \cite{PP}). However, there are families of simplicial complexes where some modification of the initial simplicial complex yields a polynomial with unimodal coefficients. This can be translated to a unimodality result involving $h$-vectors.
	
	\begin{cor}  \label{unimodsimpchrom}
		Let $S$ be a simplicial complex on $[n] = \{ 1, \ldots, n \}$ with minimal nonfaces of the form $\{ i, j, n \}$ for some $1 \le i < j \le n - 1$. There is a simplicial complex $S'$ with the following properties: 
		
		\begin{itemize}
			\item The minimal nonfaces are of the form $\{ a, b, N \}$ with $a, b \in [N - 1]$ for some $N \ge n$.
			
			\item The set of pairs $\{ i, j \}$ corresponding to the minimal nonfaces of $S$ is contained in the set of pairs $\{ a, b \}$ corresponding to the minimal nonfaces of $S'$.
			
			\item The polynomial $\frac{ \left( \frac{1 + 2u}{u} \right)^{d' - 1} }{\left( \frac{1 + u}{u} \right)^{d'} } h_{T(S')} \left( \frac{u}{1 + u} \right)$ is unimodal, where $T(S')$ is the simplicial complex with minimal nonfaces $\{ a, b \}$ from the second bullet and $d' = \dim S'$.
		\end{itemize}
	\end{cor}
	
	\begin{proof}
		Let $A$ be the simplicial complex on $[n - 1]$ with minimal nonfaces of the form $\{ i, j \}$ for each minimal nonface $\{ i, j, n \}$ of $S$. Then, we can set $A = T(S)$ in the context of Theorem \ref{coverchromhilb}. Let $G$ be the graph on $[n - 1]$ with edges corresponding to the minimal nonfaces of $A$ and $I(G, x)$ be the independence polynomial of $G$. By a result of Zhang--Dong (Theorem 1 on p. 140 of \cite{ZD}), we have that $\chi_c(S)(t) = t(t - 1)^n I \left( G, \frac{1}{t - 1} \right)$, where $I(G, x)$ is the independence polynomial of the graph $G$. In the statement of this corollary, we set $u = \frac{1}{t - 1}$ and this means that $t = \frac{1 + u}{u}$. The connection to unimodality comes from a result of Brown--Cameron (Theorem 2.4 on p. 1140 of \cite{BC}), which states that \emph{any} graph can be modified into one with a unimodal independence polynomial after sufficiently many leafy extensions. Let $B$ be the collection of pairs $\{ a, b \}$ with $a, b \in [N - 1]$ for some $N \ge n$ corresponding to the edges of this large leafy extension of $G$. Writing $u = \frac{1}{t - 1}$ and applying Theorem \ref{coverchromhilb} gives the unimodality statement in the third bullet above.
	\end{proof}

	\color{black}

		Department of Mathematics, University of Chicago \\
		5734 S. University Ave, Office: E 128 \\ Chicago, IL 60637 \\
		\textcolor{white}{text} \\
		Email address: \href{mailto:shpg@uchicago.edu}{shpg@uchicago.edu} 
\end{document}